\documentclass[12pt]{amsart}

 \usepackage{graphicx}
\usepackage{amsmath,amssymb} 
\usepackage{mathtools}
\usepackage{amsthm}
\usepackage{mathrsfs}
\usepackage[all]{xy} 
\usepackage{color}
\usepackage{multicol}
\usepackage{pgf,tikz-cd}

\allowdisplaybreaks

\theoremstyle{plain}
\newtheorem{thm}{Theorem}
\newtheorem{cor}{Corollary}
\newtheorem{lem}{Lemma}

\theoremstyle{definition}

\theoremstyle{remark}
\newtheorem{rem}{Remark}

\newcommand\C{\mathbb{C}}
\newcommand\Q{\mathbb{Q}}

\newcommand\pr{\mathbb{P}}
\newcommand\F{\mathbb{F}}
\newcommand\A{\mathbb{A}}
\newcommand\Z{\mathbb{Z}}

\DeclareMathOperator{\Hom}{Hom}
\DeclareMathOperator{\End}{End}
\DeclareMathOperator{\Spec}{Spec}

\newif\ifpersonalnote \personalnotetrue

\newcommand{\personalnoteON}{%
  \ifpersonalnote\personalnotetrue\fi}

\personalnoteON


\title[Polarization of the autonomous 4-dim
  Painlev\'e systems]{Uniqueness of polarization for the autonomous 4-dimensional
  Painlev\'e-type systems}
\author{Akane Nakamura, Eric Rains}
\date{}

\subjclass[2010]
{
    Primary: 34M55 Secondary: 33E17, 11J95.}
 \keywords{ \textit{ Painlev\'e-type equation, spectral curve, Painlev\'e divisor, abelian surface}}

\begin{document}

\begin{abstract}
  We prove that for any autonomous 4-dimensional integral system of
  Painlev\'e type, the Jacobian of the generic spectral curve has a unique
  polarization, and thus by Torelli's theorem cannot be isomorphic as an
  unpolarized abelian surface to any other Jacobian.  This enables us to
  identify the spectral curve and any irreducible genus two component of
  the boundary of an affine patch of the Liouville torus.
\end{abstract}

\maketitle


\section{Introduction}
One of the crucial aspects of the type of integrable system we consider is
that even though they are nonlinear, they possess a linear problem, usually
called the Lax pair.  Once we know the Lax pair, other nice features of the
integrable systems such as constants of motion, symmetry, and solutions in
terms of theta functions, are at our disposal.  Furthermore, one can use
qualitative properties of the Lax pair to classify such integrable
systems~\cite{2013arXiv1307.4033R}, see also
\cite{MR3087954,2016arXiv160803927K,MR3829183,MR3740334,kns}.  The question
we want to address in this paper is what if we do not know a linear problem
of an integrable system in advance.  Can we construct a Lax pair for the
given nonlinear integrable system?  The ways Lax pairs have been
constructed for finite-dimensional integrable systems seem to be more
haphazard than being systematic.  Our goal is to present a systematic way
to give a linear problem at least conceptually.  The key for doing this is
to compare curves appearing from the nonlinear side (Painlev\'e divisors)
to those appearing from the linear side (spectral curves).

There is a way to find the Laurent series solutions for weighted-homogeneous systems~\cite{MR1554772, MR737804}.
Using these Laurent series solutions, Adler and van
Moerbeke~\cite{MR999312} considered affine patches of the Liouville tori.
They call the corresponding divisor at infinity the ``Painlev\'e divisor''.
For the 4-dimensional integrable systems we consider, each component of the
Painlev\'e divisor is a curve parameterized by the constant of motions.
These are the curves we want to compare with the spectral curves.  The goal
of this paper is to show that for each of the 40 types of 4-dimensional
autonomous Painlev\'e-type system, any genus two component of the generic
Painlev\'e divisor is (geometrically) isomorphic to the corresponding
spectral curve.  Given the spectral curve (and in particular its generic
degeneration at infinity), it is straightforward to recover both the
integrable system and a corresponding linear problem.

Let us review an example from \cite{nakrims}.
The autonomous Garnier system of type $\frac{9}{2}$  is a Hamiltonian system with the following Hamiltonians~\cite{MR1042827,MR3740334}
\begin{align*}
H_1=H_{\mathrm{Gar}, s_1}^{\frac{9}{2}}
=&
p_1 q_2^2-p_1 s_1+p_2 s_2+p_1^4+3 p_2
   p_1^2+p_2^2-2 q_1 q_2,
   \\
    H_2=H_{\mathrm{Gar}, s_2}^{\frac{9}{2}}
   =&p_1^2 q_2^2-2 p_1 q_1 q_2+p_2 q_2^2+p_1^3
      s_2+p_1 s_2^2+p_2 p_1 s_2+p_2 s_1
      \\
      &-p_2p_1^3-2 p_2^2 p_1
      -q_2^2 s_2+q_1^2,
\end{align*}
where $s_1$, $s_2$ are constants.
As explained in \cite{nakrims}, the Hamiltonian system $H_{\mathrm{Gar}, s_1}^{\frac{9}{2}}$ has the following three parameter Laurent series solution,
\begin{align*}
q_1(t)
=&
-\frac{1}{t^5}+\frac{\alpha }{t^3}+\beta
   +\frac{1}{70} t \left(-35 \alpha ^3-18
   \alpha  s_2+10 s_1\right)-\frac{15}{2} t^2
   (\alpha  \beta )+\gamma 
   t^3\\
   &+O\left(t^4\right),
\\
p_1(t)
=&
\frac{1}{t^2}+\frac{\alpha }{2}+t^2
   \left(-\frac{3 \alpha ^2}{4}-\frac{3
   s_2}{5}\right)-4 \beta  t^3+t^4
   \left(-\frac{5 \alpha ^3}{4}-\frac{6 \alpha 
   s_2}{7}+\frac{s_1}{7}\right)
   \\
   &-3 t^5 (\alpha 
      \beta )+t^6 \left(-\frac{\alpha
   ^4}{8}+\frac{\gamma }{11}+\frac{3 \alpha ^2
   s_2}{385}+\frac{2 \alpha  s_1}{77}+\frac{18
   s_2^2}{275}\right)+O\left(t^7\right),
\\
q_2(t)
=&
-\frac{1}{t^3}+t \left(-\frac{3 \alpha
   ^2}{4}-\frac{3 s_2}{5}\right)-6 \beta 
   t^2+t^3 \left(-\frac{5 \alpha
   ^3}{2}-\frac{12 \alpha  s_2}{7}+\frac{2
   s_1}{7}\right)
   \\
   &-\frac{15}{2} t^4 (\alpha 
      \beta )+t^5 \left(-\frac{3 \alpha
   ^4}{8}+\frac{3 \gamma }{11}+\frac{9 \alpha
   ^2 s_2}{385}+\frac{6 \alpha 
   s_1}{77}+\frac{54
   s_2^2}{275}\right)+O\left(t^6\right),
\\	
p_2(t)
=&
-\frac{3 \alpha }{2 t^2}+\left(\frac{3 \alpha
   ^2}{2}+s_2\right)+6 \beta  t+t^2
   \left(\frac{9 \alpha ^3}{8}+\frac{9 \alpha 
   s_2}{10}\right)
   \\
   &+t^4 \left(\frac{15 \alpha
   ^4}{32}+\frac{9 \gamma }{22}-\frac{9 \alpha
   ^2 s_2}{308}-\frac{15 \alpha 
   s_1}{154}-\frac{27
   s_2^2}{110}\right)+O\left(t^5\right),
\end{align*}
where $\alpha, \beta, \gamma$ are constants.
If we restrict the Laurent series to the fiber over a point $(h_1,h_2)$, we obtain a component of the Painlev\'e divisor.
The equalities
 $H_i(q_1(t),p_1(t),q_2(t),p_2(t))=h_i, i=1,2$,
 are equivalent to the following
\begin{align*}
h_1=&
\frac{405 \alpha ^4}{32}+\frac{81 \gamma
   }{22}+\frac{648 \alpha ^2 s_2}{77}-\frac{150
   \alpha  s_1}{77}-\frac{23 s_2^2}{110},
   \\
h_2=&
\frac{729 \alpha ^5}{64}+\frac{243 \alpha 
   \gamma }{44}+81 \beta ^2+\frac{1539 \alpha
   ^3 s_2}{616}-\frac{207 \alpha ^2
   s_1}{308}-\frac{729 \alpha  s_2^2}{220}+s_1
   s_2.   
\end{align*}
By eliminating $\gamma$, the equation becomes
\begin{align}
\label{eq:Gar9/2pdivi}
  -\frac{243 \alpha ^5}{32}+81 \beta ^2+\frac{3
     \alpha  h_1}{2}-h_2-\frac{81 \alpha ^3
     s_2}{8}+s_1 \left(\frac{9 \alpha
     ^2}{4}+s_2\right)-3 \alpha  s_2^2=0. 
\end{align}
If we take $x=\frac{3}{2}\alpha, y=9\beta$, the equation becomes
\begin{align}
\label{eq:Gar9/2pdiv}
y^2=
x^5
+3 s_2 x^3
-s_1x^2
+(2 s_2^2-h_1) x
+h_2-s_1 s_2.
\end{align}

On the other hand, the Hamiltonian system $H_{\mathrm{Gar},s_i}^{\frac{9}{2}}$ has the  following Lax pair~\cite{MR3740334}.
\begin{align*}
A(x)\ =&A_0x^3+A_1x^2+A_2x+A_3,
\end{align*}
where
\begin{align*}
&A_0=
\begin{pmatrix}
0 & 1
\\
0 & 0
\end{pmatrix},
\
A_1=
\begin{pmatrix}
0 & p_1
\\
1 & 0
\end{pmatrix},
A_2=
\begin{pmatrix}
q_2 & p_1^2+p_2+2s_1
\\
-p_1 & -q_2
\end{pmatrix},
\\
&A_3=
\begin{pmatrix}
q_1-p_1q_2\ & \ p_1^3+2p_1p_2-q_2^2+s_1p_1-s_2
\\
-p_2+s_1 & -q_1+p_1q_2
\end{pmatrix}.
\end{align*}
	The spectral curve
	\begin{align*}
	\det\left(y I_2-A(x)\right)=0.
	\end{align*}
	 of the Garnier system of type $\frac{9}{2}$ is expressed as
\begin{align*}
y^2=x^5+3s_2x^3-s_1x^2+(2s_2^2-h_1)x+h_2-s_1s_2.
\end{align*}
Equation (\ref{eq:Gar9/2pdiv}) expressing the component of the Painlev\'e divisor is exactly  same as the   spectral curve of Garnier system of type $\frac{9}{2}$.
In particular, the generic degenerations of both of these genus 2 curves are of type ${\mathrm{VII}}^{*}$ in Namikawa-Ueno's notation~\cite{MR0369362}.

The conclusion we want to prove in this paper is the outcome of this example is not a coincidence.
That is, any genus 2 component of the Painlev\'e divisor is isomorphic to the corresponding spectral curve.

\subsubsection*{Acknowledgement}
\noindent 
The author would like to thank
Professors
Tadashi Ashikaga,
Brian Conrad,
Kazuki Hiroe,
Yusuke Nakamura,
Elena Mantovan,
and Haruo Yoshida for valuable discussions and advices.  The second author
was partially supported by a grant from the National Science Foundation,
DMS-1500806. The first author is grateful for Professors Toshio Oshima
(with JSPS grant-in-aid for scientific research B, No.25287017) and Yoko
Umeta (with Josai University Presedent's Fund) for supporting part of her
trips.  The first author is grateful for the hospitality of Caltech while
she stayed there.

\section{}
\subsection{Translation of the problem}
Our goal in this paper is to prove uniqueness of the isomorphism class of
genus 2 curves in the generic Liouville torus for the 4-dimensional
Painlev\'e-type systems.  We will prove this by using the fact that the
generic Liouville torus is the Jacobian of the corresponding spectral
curve.

The Jacobian of a smooth projective curve $C$ of genus $g$, 
$J(C)\coloneqq H^0(\omega_C)^{*}/H_1(C,\Z)$
comes with the canonical principal polarization $\Theta$ induced by the symplectic basis for $C$.
The classical Torelli's theorem states:
\begin{lem}[Torelli]
Two Jacobians $(J(C),\Theta)$ and $(J(C'),\Theta')$ of smooth curves $C$ and $C'$ are isomorphic as polarized abelian varieties if and only if $C$ and $C'$ are isomorphic.
\end{lem}

In general, any genus 2 curve in an abelian surface induces a principal
polarization, and thus, it will suffice to show that the Jacobian of
the generic curve in each family has a unique principal polarization.
We in fact show something stronger: any polarization of the generic
Jacobian is a multiple of the standard principal polarization.

A polarization on an abelian variety $A$ is the first Chern class $H =c_1(L)$ of a positive definite line bundle $L$ (i.e., $H =c_1(L)$ is a positive definite Hermitian form) on $A$.
 Let $t_{a}\colon A\to A$ be a translation by $a\in A$.
 Then
  $ \phi_L\colon A\to \hat{A},\
   a\mapsto t_{a}^{*}L\otimes L^{-1}$
   is a homomorphism to the dual torus $\hat{A}\simeq
   \operatorname{Pic}^0(A)$.  Using this homomorphism, the following fact
   allows us to translate the uniqueness of polarization to the triviality
   of the symmetric part\footnote{An endomorphism is symmetric when it is
     invariant under the Rosati involution: $f\mapsto
     \phi_L^{-1}\hat{f}\phi_L$ for $L\in \operatorname{NS}(A)$, where
     $\hat{f}\colon \hat{A}\to\hat{A}$ is the dual homomorphism of $f$.}
   of the endomorphism ring of $A$.
   \begin{lem}
Let $L_0$ be a principal polarization of an abelian variety $A$.
 Then we have  an isomorphism of groups
$\varphi\colon \operatorname{NS}(A)\to\operatorname{End}^s(A),\
L\mapsto \phi_{L_0}^{-1}\phi_{L},$
where $\operatorname{NS}(A)$ is the N\'eron-Severi group of $A$ and $\operatorname{End}^s(A)$, the group of symmetric endomorphisms.
   \end{lem}

\subsection{The statements of the results}
Let us state the key statement we are going to prove in this paper.
\begin{thm}\label{thm:key}
For the 4-dimensional autonomous Painlev\'e-type equations, the Jacobian of
the generic spectral curve has no nontrivial endomorphism.
\end{thm}
Using the isomorphism between $\operatorname{NS}(A)$ and the group of symmetric endomorphisms, we have
\begin{cor}
For the 4-dimensional autonomous Painlev\'e-type equations, the Jacobian of
the generic spectral curve has a unique principal polarization.
\end{cor}
Applying Torelli's theorem, we obtain the following.
\begin{thm}
For the 4-dimensional autonomous Painlev\'e-type equations, any genus 2
curve in the Jacobian of the generic spectral curve is isomorphic to the
spectral curve.
\end{thm}
We can then apply this to Painlev\'e divisors to obtain our desired result.
\begin{cor}
For the 4-dimensional autonomous Painlev\'e-type equations, any genus two
component of the generic Painlev\'e divisor is isomorphic to the
corresponding spectral curve.  In particular, the generic
degeneration\footnote{The ``generic degeneration'' of a family of curves
  over $\pr^n$ is the special fiber of the minimal proper regular model of
  the restriction to a generic line in the base.} of the spectral curve and
the generic degeneration of any irreducible component of the Painlev\'e
divisor are the same.
\end{cor}

\begin{rem}
For each of the 40 types of autonomous 4-dimensional Painlev\'e-type
equation, the generic degeneration of the spectral curve is
known~\cite{nakdoctor}.  Therefore, by computing the generic degeneration
of a genus 2 component of the Painlev\'e divisors for one of these
equations, we can identify the type.  (In particular, the 40 types of
equations, originally classified via the Lax pair, are indeed distinct as
integrable systems.)
\end{rem}

\begin{rem}
For general curves and their Jacobians, similar results  may not always hold.
For instance, Howe~\cite{MR2156657} gave families of pairs of non-isomorphic
curves whose Jacobians are isomorphic to one another as unpolarized abelian
varieties.  Note that Torelli's Theorem still applies, and thus any such
example admits multiple polarizations, so has nontrivial endomorphism
ring.  In this context, it is worth noting that although the approach below
using Igusa invariants will be difficult to extend to higher genus cases,
the semicontinuity-based approach should work more generally.
\end{rem}

\subsection{Proof of Theorem~\ref{thm:key}}
We first prove the triviality of the endomorphism rings for the most
degenerate cases ($H^{\frac{9}{2}}_{\rm {Gar}}$, $H^{\frac{5}{2}+\frac{3}{2}}_{\rm
  {Gar}}$, $H^{\frac{4}{3}+\frac{4}{3}}_{\rm{KFS}}$,
$H^{\frac{3}{2}+\frac{5}{4}}_{\rm{KSs}}$. $H^{\rm{III}(D_8)}_{\rm {Mat}}$,
$H^{\rm{I}}_{\rm {Mat}}$), since we know that any other case degenerates to
one of these 6 cases\cite{sakai_imd, kns, 2016arXiv160803927K, MR3740334,
  MR3829183}.  That this suffices follows from Lemma \ref{lem:end2} below.

\subsubsection*{Proof for $H^{\frac{9}{2}}_{\mathrm {Gar}}$,  $H^{\frac{5}{2}+\frac{3}{2}}_{\mathrm {Gar}}$,  $H^{\mathrm{III}(D_8)}_{\mathrm {Mat}}$, $H^{\mathrm{I}}_{\mathrm {Mat}}$}
We use the fact that the Jacobian of a generic hyperelliptic curve has
trivial endomorphism ring.

We will show that the family of spectral curves of a system of type
$H^{\frac{9}{2}}_{\mathrm {Gar}}$, $H^{\frac{5}{2}+\frac{3}{2}}_{\mathrm
  {Gar}}$, $H^{\mathrm{III}(D_8)}_{\mathrm {Mat}}$,
$H^{\mathrm{I}}_{\mathrm {Mat}}$ dominates the moduli space of genus two
curves, so that a typical curve in our family has no non-trivial
endomorphisms\footnote{This is not the case for
  $H_{\mathrm{KFS}}^{\frac{4}{3}+\frac{4}{3}}$ and
  $H_{\mathrm{KSs}}^{\frac{3}{2}+\frac{5}{4}}$, so we will apply a
  different proof for these two cases.}.  The moduli space of genus two
curves $\mathcal{M}_2$ can be identified with
$\operatorname{Proj}\C[J_2,J_4,J_6,J_{10}]\setminus\{J_{10}=0\}$, where
$J_{2i}$'s are the Igusa invariants and $J_{10}$ is the discriminant.
We can take the absolute invariants 
\begin{align*}
I_1=\frac{J_4}{J_2^2},\
I_2=\frac{J_6}{J_2^3},\
I_3=\frac{J_{10}}{J_2^5}
\end{align*}
as coordinates for the affine subset $\mathcal{M}_2\setminus\left\{J_2\neq
0\right\}$.  For the 4 cases ($H^{\frac{9}{2}}_{\mathrm {Gar}}$,
$H^{\frac{5}{2}+\frac{3}{2}}_{\mathrm {Gar}}$,
$H^{\mathrm{III}(D_8)}_{\mathrm {Mat}}$, $H^{\mathrm{I}}_{\rm {Mat}}$), we
can check, using the Jacobian criterion, that the Igusa invariants of their
spectral curves are algebraically independent (see Appendix), so that each
space of spectral curves dominates $\mathcal{M}_2$.  Therefore, the
Jacobian of the generic spectral curve of these 4 cases has trivial
endomorphism ring.

\subsubsection*{
Proof for $H^{\frac{4}{3}+\frac{4}{3}}_{\rm{KFS}}$,
$H^{\frac{3}{2}+\frac{5}{4}}_{\rm{KSs}}$} In these cases, the absolute
Igusa invariants $I_1, I_2, I_3$ are no longer algebraically independent,
so the previous proof cannot be applied.
Instead, we will use semicontinuity results  to reduce the problem to
curves over finite fields, where the endomorphism ring is no longer
trivial, but is often easy to compute.


The following well-known result implies that given a family $A/S$ of
abelian varieties over a reasonable scheme $S$, the endomorphism ring of
the generic fiber is the same as the endomorphism ring of the family as a
whole.

\begin{lem}
\label{lem:end1}
 Let $S$ be an integral normal Noetherian scheme, let $X$ be a smooth 
$S$-scheme, and let $A$ be an abelian scheme over $S$.
  Then any morphism 
$X_{K(S)}\to A_{K(S)}$ 
extends to a unique morphism $X\to A$.                                               
In particular, the natural map 
$\End_S(A)\to \End_{K(S)}(A_{K(S)})$
 is an isomorphism\footnote{$K(S)=\mathcal{O}_{S,\xi}$, where $\xi$ is the generic point of  an irreducible scheme $S$, denotes the function field of $S$.}. 
\end{lem}                              
                                              
\begin{proof}
 Any morphism of 
$X_{K(S)}\to A_{K(S)}$ determines a rational map $\phi:X\dashrightarrow A$.
 $X$ is normal, since it is smooth over $S$, and $A\to S$ is proper over
 $S$ by definition.\footnote{Abelian schemes are proper and smooth
   $S$-group schemes with connected fibers} Therefore, by the valuative
 criterion for properness, $\phi$ extends uniquely to the (regular) local
 ring over any codimension 1 point.  It follows by Weil's extension theorem
 that $\phi$ is defined 
everywhere, so gives a morphism as required.

In particular, an endomorphism of $A_{K(S)}$ extends as a morphism of
schemes from $A$ to itself.  A morphism of abelian schemes is a
homomorphism if and only if it fixes the identity, and this follows for
$\phi$ since it fixes the generic point of the identity section.

The isomorphism claim follows from the fact that $\phi\in \End(A)$ is
trivial on the generic fiber if and only if the image of the generic fiber is the
identity, which implies that the image of $\phi$ itself is the identity, so
that $\phi$ is trivial.
\end{proof}                 
\begin{rem}
This is just Prop.1.2.8 of \cite{MR1045822} when $S$ is Dedekind, saying
$X$ is a N\'eron model of its generic fiber $X_{K(S)}$.  The proof is
essentially the same, but we should replace ``closed point of $S$" in the
Dedekind case by ``codimension 1 point of $S$" when $S$ is not necessarily
of dimension 0 or 1.  The result is widely known, but we have added a proof
since we could not find a reference.
\end{rem}    

In order to study the endomorphism ring over $S$, we can use various
specializations due to the following well-known injectivity result for
endomorphism rings of Abelian schemes.
\begin{lem}\label{lem:end2}
 Let S be a connected normal scheme, and let $A/S$, $B/S$ be abelian 
schemes over $S$.
  Then for any point $s\in S$, the natural map $\Hom_S(A,B)\to 
\Hom_{k(s)}(A_s,B_s)$ is injective\footnote{$k(s)=\mathcal{O}_S,s/\mathfrak{m}_s$ denotes the residue field.}.
\end{lem}

Thus, to show that $\End_{K(S)}(A_{K(S)})$ is trivial, or equivalently that
$\End_S(A_S)$ is trivial, it suffices to show that $\End_s(A_s)$ is trivial
for some point $s\in S$.  In our cases, the family $S$ is defined over a
number field ($\Q$, in fact), and we may take $s$ to be a point over that
field, so that we reduce to showing that a particular abelian scheme over a
number field has trivial endomorphism ring.  If we try to specialize any
further, we encounter the difficulty that an abelian scheme over a finite
field always has complex multiplication, so nontrivial endomorphism ring.
However, we can hope to find {\em two} specializations such that their
endomorphism rings do not have any isomorphic subrings other than $\Z$.

In our case, the natural family has $S=\A^3_{\Q}$ (or, more precisely, the
complement of the discriminant locus in the affine space).  If the
endomorphism ring of the family is indeed trivial, then we expect that {\em
  most} fibers over $\Q$ will have trivial endomorphism ring, so we simply
choose such points until we find one that works.  We then need only find a
pair of primes $p\in \Spec\Z$ for which we can compute the endomorphism
rings and verify the lack of ``intersection''.


%

\begin{figure}[h]
\includegraphics[width=12cm, natwidth=2535, natheight=828]{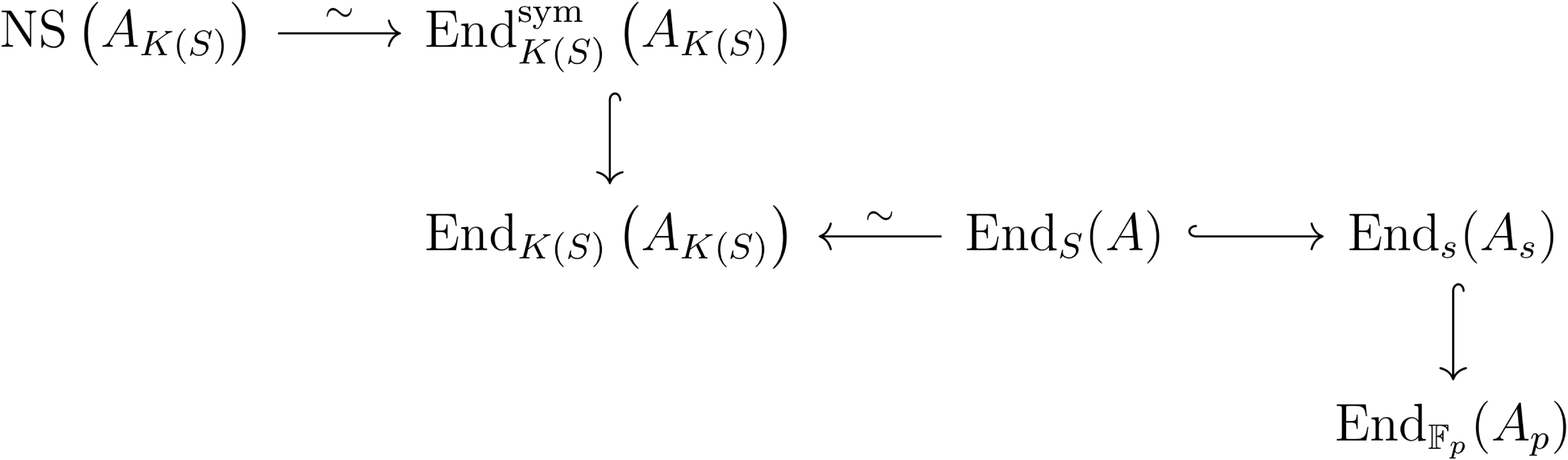}
\end{figure}	


    
The endomorphism of an abelian variety over a finite field $\F_p$ can be
studied using
the Frobenius endomorphism\footnote{For a projective variety over $\F_p$, the Frobenius morphism is the map on points $(x_0:\cdots:x_n)\mapsto(x_0^p:\cdots:x_n^p)$}.
%
For the Jacobian $A=J(C)$ of a curve $C$, the Frobenius endomorphism $\pi_{C_p}$ of the curve $C$ induces the Frobenius endomorphism $\pi_p$ on $A_p$.
Through the correspondence of the \'etale cohomology groups $H^1_{et}(A_p,\Z_l)$ and $H^1_{et}(C_p,\Z_l)$,
we know that the characteristic polynomial  $ P$ of $\pi_{C_p}$ is same as the characteristic polynomial of $\pi_p$ acting on the Tate module $T_l(A)$, where $l$ is a prime different from $p$.

 The characteristic polynomial $ P$ of $\pi_{C_p}$ is computable via the
 zeta function.  The zeta function of the curve $C$ defined over $\F_{p}$
 is given by the formula
 \begin{align*}
 Z(C,s)=\exp\left(\sum_{m=1}^{\infty}\frac{\# C(\F_{p^m})}{m}p^{-ms}\right).
 \end{align*} 
 According to the  Weil conjectures~\cite{MR0029522},
  $Z(C,s)$ is a rational function of $t=p^{-s}$ and
   we can write
   \begin{align*}
   Z(C,s)=\frac{L(t)}{(1-t)(1-pt)}\eqqcolon Z_C(t)
   \end{align*}
for a certain polynomial $L(t)$ of degree $2g$ with integer coefficients. 
 For genus $g=2$ case, the numerator is
 \begin{align*}
 L(t) = p^2t^4-pa_1t^3 + a_2t^2-a_1t+1,
 \end{align*}
 where 
 \begin{align*}
 a_1=p+1-N_1,\ a_2= \frac{1}{2}(N_2 +N_1^2)-(p+1)N_1 +p
 \end{align*}
 and $N_m=\# C(\F_{p^m})$.  In particular, for small $p$, this can be
 efficiently computed by enumerating points on any smooth projective model
 of $C$.  The characteristic polynomial of the Frobenius endomorphism
 $\pi$ of the Jacobian $J(C)$ is then essentially given by the numerator of
 the zeta function, or more precisely
\begin{align*}
P(t)
=t^{4}L\left(\frac{1}{t}\right)
=t^4-a_1t^3+a_2t^2-pa_1t+p^2.
\end{align*}
When the characteristic polynomial $P(t)$ has no multiple roots over $\C$, nonzero discriminant, then Theorem~2 of Tate~\cite{MR0206004} tells us that the rational endomorphism ring $\operatorname{End}_{\F_p}(J(C))\otimes\Q$ is generated by Frobenius,
i.e.
$\operatorname{End}_{\F_p}(A)\otimes \Q\cong \Q[t]/P(t)$.  (More generally, if the $l$-th powers of the roots are distinct, then $\operatorname{End}_{\F_{p^l}(A)}\otimes\Q$ is generated by the $l$-th power of Frobenius, so there are no additional endomorphisms defined over $\F_{p^l}$.)
 
Let us consider the case $H^{\frac{4}{3}+\frac{4}{3}}_{\rm{KFS}}$.
The spectral curve for $H^{\frac{4}{3}+\frac{4}{3}}_{\rm{KFS}}$ is
\begin{align*}
y^2=&
x^6-2 x^5+(2h_1+1)x^4
+2( h_2 -h_1)x^3
+(h_1^2 -2h_2)x^2
   +2 h_1 h_2 x
   \\
   &+h_2^2-4 s.
\end{align*} 

Consider an instance
 $h_1=12$, $h_2=17$, $s=29$ and reduce this curve modulo $p=37$.
\begin{align*}
C_1\colon
y^2=x^6 + 35x^5 + 25x^4 + 10x^3 + 36x^2 + x + 25.
\end{align*}
We can compute using Magma\cite{MR1484478} that $\# C_1\left(\F_{p}\right)=36$, $\# C_1\left(\F_{p^2}\right)=1442$, 
The zeta function of this hyperelliptic curve is
\begin{align*}
Z_{C_1}(t)
=
\frac{37^2t^4 - 37\cdot 2 t^3 + 38t^2 - 2t + 1}{(1-t)(1-37t)}.
\end{align*}
The characteristic polynomial of Frobenius is 
\begin{align*}
P_1(t)=&t^4-2 t^3+38 t^2-37\cdot 2 t+37^2
\\
=&\left(t^2-(1+\sqrt{37})t+37\right)\left(t^2-(1-\sqrt{37})t+37\right).
\end{align*}
The quartic is square-free and is
irreducible over $\Q$.  Therefore,
$E_1\coloneqq\operatorname{End}_{\F_p}(J(C_1))\otimes\Q \simeq \Q(\alpha)=\Q[t]/P_1(t)$ is a degree $2g=4$ extension of $\Q$, where $\alpha$ is one of the roots of $P_1(t)$.
Since $P_1(t)$ is irreducible and the Galois group of $P_1(t)$ contains a transposition $(13)\in S_4$, but does not contain the whole group $S_4$, the Galois group is dihedral group
$D_4=\left<\sigma=(1234),\tau=(13)\right>$.
Each order 2 subgroup of $D_4$ is contained in a unique subgroup of order
4.  Therefore, Galois theory tells us that $\Q(\alpha)$ contains unique
subfield of degree 2 over $\Q$, namely
$\Q\left(\alpha+\frac{p}{\alpha}\right)\simeq \Q(\sqrt{37})$.
Consider the same curve over $\Q$ but now reduce modulo a different prime
$p=53$ to obtain
\begin{align*}
C_2\colon
y^2=x^6 + 51x^5 + 25x^4 + 10x^3 + 4x^2 + 37x + 14.
\end{align*}
We can again compute using Magma that $\# C_2\left(\F_{p}\right)=57$, $\# C_2\left(\F_{p^2}\right)=3001$.
The zeta function of this hyperelliptic curve is
\begin{align*}
Z_{C_2}(t)
=
\frac{53^2t^4 + 53\cdot 3t^3 + 100t^2 + 3t + 1}{(1-t)(1-53t)}.
\end{align*}
The characteristic polynomial of Frobenius is 
\begin{align*}
P_2(t)
=&t^4+3 t^3+100 t^2+53\cdot 3 t+53^2
\\
=&\left(t^2+\frac{3+\sqrt{33}}{2}t+53\right)\left(t^2+\frac{3-\sqrt{33}}{2}t+53\right).
\end{align*}
We again find that $E_2\coloneqq\operatorname{End}_{\F_p}(J(C_2))\otimes
\Q\simeq \Q(\beta)\simeq\Q[t]/P_2(t)$ is a degree $2g=4$ extension of $\Q$,
where $\beta$ is one of the roots of $P_2(t)$.  $\Q(\beta)$ again contains
a unique subfield of degree 2 over $\Q$, namely
$\Q\left(\beta+\frac{p}{\beta}\right)\simeq \Q(\sqrt{33})$.
 
 From the semi-continuity of the reduction map, if
 $\operatorname{End}(A)\otimes \Q$ were strictly larger than $\Q$, it would
 inject in both reductions, and thus the fields $E_1$ and $E_2$ would have
 a common subfield other than $\Q$.  Since the fields are not isomorphic,
 that common subfield would have to be quadratic, and we have already
 determined that their respective unique quadratic subfields are not
 isomorphic, which is what we needed to prove.
 
\begin{rem}
  This only proves that the endomorphism ring over $k(S)$ is trivial, but in fact the same argument easily shows that the endomorphism ring remains trivial over $\overline{k(S)}$, and thus in particular, any {\em geometric} genus 2 component of the Painlev\'e divisor is isomorphic to the spectral curve.  The only thing that needs to be checked is that no two roots of the respective characteristic polynomials of Frobenius have ratio a root of unity, since then that implies that the {\em geometric} endomorphism ring is contained in the relevant field.
\end{rem}
 \begin{rem}
 The approach  we have used here is similar to the argument  of section ``The structure of $\End(A)$" in Appendix A of \cite{MR2058652}.
 \end{rem}
 
 For the $H^{\frac{3}{2}+\frac{5}{4}}_{\rm{KSs}}$, a similar argument with $h_1=12$, $h_2=17$, $s=29$ and $p=37$ and $p=31$ proves the result we wanted.
 \begin{rem}
 These proofs using reduction modulo $p$ are also applicable for the other
 cases we have proved using the algebraic independence of the Igusa
 invariants.  This approach is far more likely to work in higher genus,
 both because the analogues of Igusa invariants are much more complicated
 (when known!), and because the dimension of the moduli space of Jacobians
 grows much faster with the genus than one expects the dimensions of
 families of spectral curves to grow.
 \end{rem}
 
 \subsubsection*{Proof for all the other cases}
 All the other 34 cases degenerate to one of the 6 cases.  Let us recall
 what it means for an equation $P_A$ to degenerate to an equation $P_B$.
Let $R$ be a dvr with maximal ideal $\mathfrak{m}$ and residue field
$k=R/\mathfrak{m}$, and consider a curve $\mathcal{C}/\Spec
R[h_1,h_2,\frac{1}{\Delta}]$.  Let $I$ be the ideal defined by
\begin{align*}
0\to I\to R\left[h_1,h_2,\frac{1}{\Delta}\right]\to
 k\left[h_1,h_2,\frac{1}{\Delta}\right]\to 0.
\end{align*}
Then, $I$ is a prime ideal of codimension 1 and the localization $R_{I}$ is
a dvr.  If $P_A$ is the base change of $C$ to the field of fractions of $R$
and $P_B$ is the base change to $k$, then we say that $P_A$ degenerates to
$P_B$.

In such a case, the generic spectral curve of $P_B$ is the curve over
$R_{I}/\mathfrak{m}_I \simeq k(h_1,h_2)$, which is the special fiber of the
curve over $R_I$, the generic fiber of which is the generic spectral curve
of $P_A$.  Then Lemma~\ref{lem:end1} and Lemma~\ref{lem:end2} guarantee
that the endomorphism ring for $P_A$ injects in the endomorphism ring for
$P_B$.  Thus, if the endomorphism ring for $P_B$ is trivial, so is
$\End(P_A)$.  Since the endomorphism rings for the 6 most degenerate cases
are trivial, so are the generic instances of the other 34 cases.

This ends a proof for all 40 cases.
\qed

\appendix
\section{Spectral curves}
\subsection{The Garnier system of type $\frac{9}{2}$}
The Riemann scheme for this system, expressing the singular point, its ramification index and  diagonal elements of the Hukuhara-Turrittin-Levelt canonical form~\cite{zbMATH03026099, zbMATH03109720, zbMATH03477550} as in \cite{MR3740334} is given by
\[
\left(
\begin{array}{c}
 x=\infty \left(\frac{1}{2}\right)\\
\overbrace{\begin{array}{ccccccccc}
   1 &  0 & 0  & 0 & \frac{3}{2}s_1 & 0 & -\frac{s_2}{2} & 0\\
   -1 &  0 & 0 & 0 & -\frac{3}{2}s_1 & 0 & \frac{s_2}{2}  & 0
        	\end{array}}
\end{array}
\right) .
\]
	The spectral curve
	\begin{align*}
	\det\left(y I_2-A(x)\right)=0
	\end{align*}
	 of the autonomous Garnier system of type $\frac{9}{2}$ is expressed as
\begin{align*}
y^2=x^5+3s_2x^3-s_1x^2+(2s_2^2-h_1)x+h_2-s_1s_2,
\end{align*}
where $h_1$ and $h_2$ are the Hamiltonians and $s_1$, $s_2$ are constants.
The Igusa invariants are
\begin{align*}
J_2=&80 h_1+12 s_1^2,
\\
J_4=&-800 h_2 s_2+480 h_1^2-16 h_1 s_1^2+32 s_2^2 s_1+6 s_1^4,
\\
J_6=&-16000 h_2^2 s_1+6400 h_2 h_1 s_2+320 h_2 s_2 s_1^2-1280
   h_1^3+704 h_1^2 s_1^2
   \\
   &
   -896 h_1 s_2^2 s_1
   -112 h_1 s_1^4
  +256
   s_2^4+64 s_2^2 s_1^3+4 s_1^6,
   \\
  J_8=&
  \frac{1}{4}(J_2J_6-J_4^2),
  \\
  J_{10}=\Delta,
\end{align*}
where $\Delta$ is the discriminant of the quintic.
The absolute invariants are as follows.
\begin{align*}
I_1=\frac{J_4}{J_2^2},\
I_2=\frac{J_6}{J_2^3},\
I_3=\frac{J_{10}}{J_2^5}.\
\end{align*}
Since the matrix
\begin{align*}
\begin{pmatrix}
\frac{\partial I_i}{\partial h_1}
 & 
 \frac{\partial I_i}{\partial h_2}
  & 
  \frac{\partial I_i}{\partial s_1}
   & 
   \frac{\partial I_i}{\partial s_2}
\end{pmatrix}_{i=1,2,3}
\end{align*}
has rank 3, from the Jacobian criterion, $I_1, I_2, I_3$ are algebraically independent.
\begin{rem}
  The generic fiber of the corresponding family of quintics has Galois
  group $S_5$, which also implies triviality of the endomorphism ring, by a
  result of Zarhin~\cite{MR1748293}.
\end{rem}

\subsection{The Garnier system of type $\frac{5}{2}+\frac{3}{2}$}
The Riemann scheme for this system is \cite{MR3740334}
\[
\left(
\begin{array}{cc}
  x=0 \left(\frac{1}{2}\right) & x=\infty  \left(\frac{1}{2}\right)\\
\overbrace{\begin{array}{cc}
     \sqrt{s_2} & 0 \\
     -\sqrt{s_2} & 0 
           \end{array}}
& 
\overbrace{\begin{array}{cccc}
     1 & 0 & -s_1/2 & 0 \\
     -1 & 0 &  s_1/2 & 0 
           \end{array}} 
\end{array}
\right).
\]
The spectral curve is
\begin{align*}
y^2
=
x^5
-s_1 x^4
+h_1 x^3+
h_2 x^2
+s_2 x.
\end{align*}
We again verify that the Jacobian matrix
\begin{align*}
\begin{pmatrix}
\frac{\partial I_i}{\partial h_1}
 & 
 \frac{\partial I_i}{\partial h_2}
  & 
  \frac{\partial I_i}{\partial s_1}
   & 
   \frac{\partial I_i}{\partial s_2}
\end{pmatrix}_{i=1,2,3}
\end{align*}
has rank 3, so the Igusa invariants $I_1, I_2, I_3$ are algebraically
independent.  (This can also be seen by inspection: any genus 2 curve can
be put into the above form by choosing a pair of Weierstrass points.)

\subsection{The first Matrix Painlev\'e system $H_{\mathrm{Mat}}^{\mathrm{I}}$}
The Riemann scheme is given by
\[
\left(
\begin{array}{c}
 x=\infty \left(\frac{1}{2}\right)\\
\overbrace{\begin{array}{cccccc}
   1 &  0 & 0  & 0 & s/2 & \theta^\infty_2/2\\
   1 &  0 & 0 & 0 & s/2 & \theta^\infty_3/2\\
   -1 & 0 & 0& 0 & s/2& \theta^\infty_2/2\\
   -1 & 0 & 0 & 0 & s/2 & \theta^\infty_3/2
        	\end{array}}
\end{array}
\right) ,
\]
and the Fuchs-Hukuhara relation is expressed as $\theta^{\infty}_2+\theta^{\infty}_3=0.$
\begin{align}
y^4
&
- \left(2x^3+2 s x+h_1\right)y^2
+x^6
+2 s x^4
+h_1x^3
+s^2 x^2
\\
&
+\left(h_1 s-(\theta _2^{\infty})^2\right)x
   +h_2=0.
   \nonumber
\end{align}
The spectral curve has the following Weierstrass form.
\begin{align*}
y^2
=
&
x^6
-3 h_1 x^5
+3( h_1^2 +h_2)x^4
-h_1^3x^3
-6 h_1 h_2 x^3
+3 (h_1^2 +h_2)h_2 x^2
\\
&
+\left(\theta_2^{\infty}\right)^4 s x^2
-h_1 \left(\theta _2^{\infty}\right)^4 s x
-\left(\theta _2^{\infty}\right)^6 x
-3 h_1 h_2^2 x
+h_2^3
+h_2 \left(\theta _2^{\infty}\right)^4 s.
\end{align*}
We find as above that the absolute invariants $I_1, I_2, I_3$ are
algebraically independent.

\subsection{The third Matrix Painlev\'e system $H_{\mathrm{Mat}}^{\mathrm{III}(D_8)}$}
The Riemann scheme is given by
\[
\left(
\begin{array}{cc}
  x=0 \left(\frac{1}{2}\right)& x=\infty \left(\frac{1}{2}\right)\\
\overbrace{\begin{array}{cc}
     \sqrt{t} & 0 \\
     \sqrt{t} & 0 \\
     -\sqrt{t} & 0 \\
     -\sqrt{t} & 0
           \end{array}}
& 
\overbrace{\begin{array}{cc}
     1 & \theta^\infty_2/2 \\
     1 & \theta^\infty_3/2 \\
     -1 & \theta^\infty_2/2 \\
     -1 & \theta^\infty_3/2
           \end{array}} 
\end{array}
\right) ,
\]
and the Fuchs-Hukuhara relation is written as
$\theta_2^\infty +\theta_3^\infty =0$.
\begin{align*}
&y^4
+2 x y^3
-\left(-2 x^3+(h_1+(\theta_2^{\infty})^2 )x^2+2 s x\right)y(x+y) 
+x^2y^2
\\
&
+x^6
+h_1 x^5
+h_2 x^4
h_1 s x^3
+(\theta_2^{\infty})^2 s x^3
+s^2 x^2
=0.
\end{align*}
We omit writing the explicit Weierstrass form of the equation, but {\tt
  Maple} allows us to compute it, and thus the Igusa invariants, so that we
can again show that the absolute invariants are algebraically independent.


\def\cydot{\leavevmode\raise.4ex\hbox{.}}

\end{document}